\documentclass[a4paper,12pt]{article}

\usepackage{latexsym}
\usepackage{amssymb}
\usepackage{theorem}
\usepackage{amsmath}
\usepackage{amscd}
\usepackage{graphicx}
\usepackage{color}

\newtheorem{theorem}{Theorem}[section]
\newtheorem{corollary}[theorem]{Corollary}
\newtheorem{lemma}[theorem]{Lemma}
\newtheorem{example}[theorem]{Example}
\newtheorem{proposition}[theorem]{Proposition}
\newtheorem{remark}[theorem]{Remark}
\newtheorem{definition}[theorem]{Definition}

\newcommand{\demo}{\par\noindent{\it Proof. \/}\ }
\newcommand{\enD}{\hfill $\Box$\vspace{3truemm} \par}
\newcommand{\R}{\mathbb{R}}

\newcommand{\ba}{\mbox{\boldmath $a$}}
\newcommand{\bb}{\mbox{\boldmath $b$}}

\newcommand{\A}{\mathcal{A}}

\begin{document}

\title{Curvature equivalence for Legendre curves in the unit tangent bundle over Euclidean plane}

\author{Nozomi Nakatsuyama, Masatomo Takahashi and Minoru Yamamoto}

\date{\today}

\maketitle

\begin{abstract}
The Legendre curve in the unit tangent bundle over Euclidean plane is a plane curve with a moving frame. 
We have the (Legendre) curvature of the Legendre curve, and the existence and uniqueness theorems for the curvature are valid. 
In this paper, we introduce an equivalence relation for Legendre curves called the curvature equivalence. 
We investigate properties of the curvature equivalence and give local and global classifications of Legendre curves under the curvature equivalence.
\end{abstract}

\renewcommand{\thefootnote}{\fnsymbol{footnote}}

\section{Introduction}

A lot of works on the classification of germs with $\mathcal{A}$-equivalence (diffeomorphisms of the source and target) have been done in the context of singularity theory. 
See \cite{Arnold1, Arnold2, Bruce-Giblin, Izumiya-book} for basic idea of singularity theory.
The $\mathcal{A}$-equivalent class of mappings or functions is worth studying from the viewpoint of differential topology. 
However, diffeomorphisms destroy geometry while preserving singularities.
In \cite{Salarinoghabi-Tari}, they investigated the geometry of deformations of singular plane
curves, which preserves inflection points, vertices and singular points of plane curves. 
In this paper, we consider another equivalence relation for smooth plane curves with singular points.
As smooth plane curves with singular points, we apply Legendre curves. 
A Legendre curve in the unit tangent bundle over Euclidean plane is a plane curve with a moving frame. 
Then we can define the (Legendre) curvature of the Legendre curve  (cf.\ \cite{Fukunaga-Takahashi-2013}). 
The curvature is a complete invariant for Legendre curves up to congruence (Euclidean motion) like as the curvature of a regular plane curve. 
The existence and uniqueness theorems for the curvature are valid (cf.\ \cite{Fukunaga-Takahashi-2013}). 
We introduce an equivalence relation for Legendre curves via its curvatures called the curvature equivalence. 
Properties of the curvature equivalence are investigated, for instance, an affine transformation of a plane curve is one example of the curvature equivalence. 
By definition, the curvature equivalence preserves the inflection points and singular points of curves but does not preserve vertices. 
As main results, we give local and global classifications of Legendre curves under the curvature equivalence (Theorems \ref{local-classifications} and \ref{global-classifications}). 
Consequently, the contact order of the components of curves is important for  the local classification of curvature equivalence. 
Furthermore, the relative positions of the types of singular and inflection points are important for the global classification of curvature equivalence. 

We shall assume throughout the paper that all maps and manifolds are $C^{\infty}$ unless the contrary is explicitly stated.

\bigskip
\noindent
{\bf Acknowledgements}. 
The first author was supported by JST SPRING Grant Number JPMJSP2153.
The second author was partially supported by JSPS KAKENHI Grant Number JP 24K06728.

\section{Preliminaries}

We quickly review the theories of Legendre curves on the unit tangent bundle over $\R^2$.
\par
Let $I$ be an interval or $\R$, and let $\R^2$ be the Euclidean plane with the inner product $\ba \cdot \bb=a_1b_1+a_2b_2$, where $\ba=(a_1,a_2)$ and $\bb=(b_1,b_2) \in \R^2$. 
\par
Let $(\gamma,\nu): I \to \R^2 \times S^1$ be a smooth mapping, 
where  $S^1$ is the unit circle centered at the origin in $\R^{2}$. 
We say that $(\gamma,\nu):I \to \R^2 \times S^1$ is {\it a Legendre curve} if 
$(\gamma,\nu)^* \theta=0$, where $\theta$ is the canonical contact $1$-form on the unit tangent bundle $T_1 \R^2=\R^2 \times S^1$ (cf.\ \cite{Arnold1, Arnold2}). 
This condition is equivalent to $\dot{\gamma}(t) \cdot \nu(t)=0$ for all $t \in I$. 
Moreover, if $(\gamma,\nu)$ is an immersion, we call $(\gamma,\nu)$ {\it a Legendre immersion}.
We say that $\gamma: I \to \R^2$ is {\it a frontal} (respectively, {\it a front} or {\it a wave front}) if there exists a smooth mapping $\nu:I \to S^1$ such that 
$(\gamma,\nu)$ is a Legendre curve (respectively, a Legendre immersion).
\par
Let $(\gamma,\nu):I \to \R^2 \times S^1$ be a Legendre curve.
We put on $\mu (t)=J(\nu (t))$ and call the pair $\{\nu(t), \mu(t) \}$ a {\it moving frame along the frontal $\gamma(t)$} in $\R^2$. 
Here, $J$ represents the $\pi/2$ anticlockwise rotation around the origin in $\R^{2}$. 
Then, we have the Frenet type formula of the Legendre curve, which is given by
\begin{eqnarray*}\label{Frenet.frontal}
\left(
\begin{array}{c}
\dot{\nu}(t)\\
\dot{\mu}(t)
\end{array}
\right)
=
\left(
\begin{array}{cc}
0 & \ell(t)\\
-\ell(t) & 0
\end{array}
\right)
\left(
\begin{array}{c}
\nu(t)\\
\mu(t)
\end{array}
\right), \ 
\dot\gamma(t)=\beta(t) \mu(t),
\end{eqnarray*}
where $\ell(t)=\dot\nu(t) \cdot \mu(t), \beta(t)=\dot{\gamma}(t) \cdot \mu(t)$. 
The pair $(\ell, \beta)$ is an important invariant of Legendre curves. 
We call the pair $(\ell,\beta)$ the {\it (Legendre) curvature} of the Legendre curve $(\gamma,\nu)$ (cf.\ \cite{Fukunaga-Takahashi-2013}).
Then $(\gamma,\nu):I \to \R^2 \times S^1$ is a Legendre immersion if and only if  $(\ell(t),\beta(t)) \not=(0,0)$ for all $t \in I$. 
We say that a point $t_0 \in I$ is an {\it inflection point} of  the frontal $\gamma$ (or, the Legendre curve $(\gamma,\nu)$) if $\ell(t_0)=0$.
Remark that the definition of the inflection point of the frontal is a generalization of the definition 
of the inflection point of a regular plane curve (cf.\ \cite{Fukunaga-Takahashi-2013}). 
Moreover, $t_0$ is a singular point of $\gamma$ if and only if $\beta(t_0)=0$.
\par
We have the existence and the uniqueness for Legendre curves in the unit tangent bundle like as regular plane curves, see \cite{Fukunaga-Takahashi-2013}.
\begin{theorem}[The Existence Theorem for Legendre curves \cite{Fukunaga-Takahashi-2013}]\label{existence.Legendre}
Let $(\ell,\beta):I \to \R^2$ be a smooth mapping. 
There exists a Legendre curve $(\gamma,\nu):I \to \R^2 \times S^1$ whose associated curvature of the Legendre curve is $(\ell, \beta)$.
\end{theorem}
\par
In fact, the Legendre curve, whose associated curvature of the Legendre curve is $(\ell,\beta)$, is given by the form
\begin{align*}
\gamma(t)&=\left(-\int \beta(t) \sin \left( \int \ell(t) dt \right) dt, \  \int \beta(t) \cos \left( \int \ell(t) dt \right) dt\right),\\
\nu(t)&=\left(\cos \left(\int \ell(t) dt\right), \ \sin \left(\int \ell(t) dt\right)\right).
\end{align*}
\begin{definition}\label{congruent}{\rm
Let $(\gamma,\nu)$ and $(\widetilde{\gamma},\widetilde{\nu}):I \to \R^2 \times S^1$ be Legendre curves.
We say that $(\gamma,\nu)$ and $(\widetilde{\gamma},\widetilde{\nu})$ are {\it congruent as Legendre curves} if there exist a constant rotation $A \in SO(2)$ and a translation $\ba$ on $\R^2$ such that $\widetilde{\gamma}(t)=A(\gamma(t))+\ba$ and $\widetilde{\nu}(t)=A (\nu(t))$ for all $t \in I$. 
}
\end{definition}
\begin{theorem}[The Uniqueness Theorem for Legendre curves \cite{Fukunaga-Takahashi-2013}]\label{uniqueness.Legendre}
Let $(\gamma,\nu)$ and $(\widetilde{\gamma},\widetilde{\nu}):I \to \R^2 \times S^1$ be Legendre curves with curvatures of Legendre curves $(\ell,\beta)$ and $(\widetilde{\ell},\widetilde{\beta})$.
Then $(\gamma,\nu)$ and $(\widetilde{\gamma},\widetilde{\nu})$ are congruent as Legendre curves if and only if $(\ell,\beta)$ and $(\widetilde{\ell},\widetilde{\beta})$ coincide.
\end{theorem}

\begin{example}[Type $(n,m)$ \cite{Nakatsuyama-Takahashi-2024}]\label{nm-type}
{\rm 
Let $n, k \in \mathbb{N}$ and $m=n+k$. 
We consider a smooth map germ $(\gamma,\nu):(\R,0) \to \R^2 \times S^1$,
\begin{align*}
\gamma(t) &=(\pm t^n,t^mf(t)),\ 
\nu(t) =\frac{1}{\sqrt{(mt^kf(t)+t^{k+1}\dot{f}(t))^2+n^2}}(-mt^kf(t)-t^{k+1}\dot{f}(t),\pm n),
\end{align*}
where $f:(\R,0) \to \R$ is a smooth function germ with $f(0) \not=0$. 
Note that $0$ is a singular point of $\gamma$ when $n>1$. 
Then $(\gamma,\nu)$ is a Legendre curve with the curvature 
\begin{align*}
(\ell(t),\beta(t))=\biggl(&\pm \frac{nt^{k-1}(mkf(t)+(m+k+1)t\dot{f}(t)+t^{2}\ddot{f}(t))}{(mt^kf(t)+t^{k+1}\dot{f}(t))^2+n^2},\\
&-t^{n-1}\sqrt{(mt^kf(t)+t^{k+1}\dot{f}(t))^2+n^2}\biggr).
\end{align*}
We say $\gamma$ is of type $(n,m)$. 
If $n=1$, then $\gamma$ is regular (a front) and $(\gamma,\nu)$ is a Legendre immersion.
If $k=1$, then $\gamma$ is a front and $(\gamma,\nu)$ is a Legendre immersion.
Otherwise, that is, if $n,k>1$, then $\gamma$ is a frontal and $(\gamma,\nu)$ is a Legendre curve. 
See Figures \ref{fig:23cusp} and \ref{fig:45cusp} for examples of types $(2,3)$ and $(4,5)$.

\begin{figure}[htbp]
\centering
\begin{minipage}{0.3\textwidth}
\centering
\includegraphics[width=4cm,height=4cm]{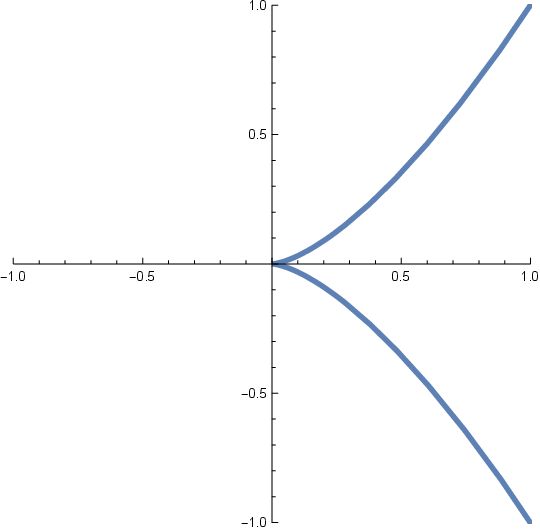}
\end{minipage}
\hspace{0.5cm}
\begin{minipage}{0.3\textwidth}
\centering
\includegraphics[width=4cm,height=4cm]{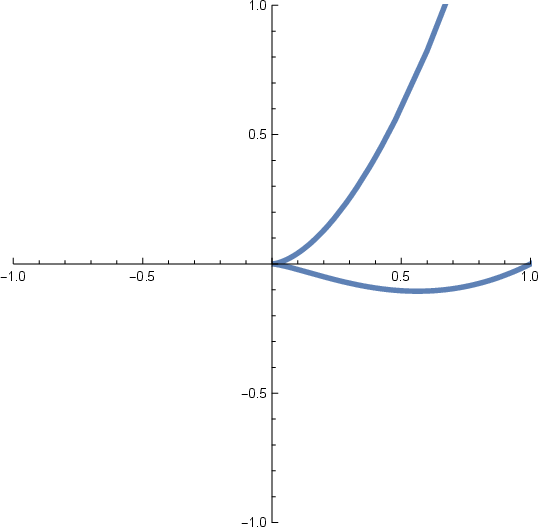}
\end{minipage}
\caption{On the left is the case of $(n,m)=(2,3)$, with a sign of $+$ and $f(t)=1$, that is, $\gamma(t)= (t^2,t^3)$. On the right is the case of $(n,m)=(2,3)$, with a sign of $+$ and $f(t)=1+t$, that is, $\gamma(t) = (t^2,t^3+t^4)$.}
\label{fig:23cusp}
\end{figure}
\begin{figure}[htbp]
\centering
\begin{minipage}{0.3\textwidth}
\centering
\includegraphics[width=4cm,height=4cm]{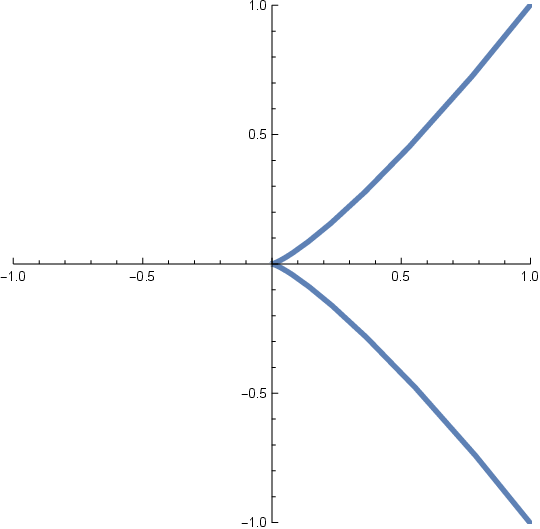}
\end{minipage}
\hspace{0.5cm}
\begin{minipage}{0.3\textwidth}
\centering
\includegraphics[width=4cm,height=4cm]{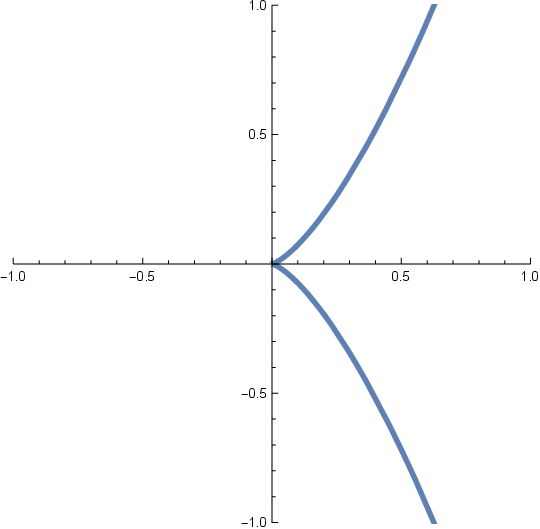}
\end{minipage}
\caption{On the left is the case of $(n,m)=(4,5)$, with a sign of $+$ and $f(t)=1$, that is, $\gamma(t)= (t^4,t^5)$. On the right is the case of $(n,m)=(4,5)$, with a sign of $+$ and $f(t)=1+t^2$, that is, $\gamma(t) = (t^4,t^5+t^7)$.}
\label{fig:45cusp}
\end{figure}
}
\end{example}

\begin{proposition}[\cite{Fukunaga-Takahashi-2013}]\label{Legendre.function}
Let $(\gamma,\nu):I \to \R^2 \times S^1$ be a Legendre curve with curvature $(\ell,\beta)$ and $t:\widetilde{I} \to I$ be a smooth function, where $\widetilde{I}$ is an interval. Then $(\gamma \circ t, \nu \circ t): \widetilde{I} \to \R^2 \times S^1$ is also a Legendre curve with the curvature $((\ell \circ t)t', (\beta \circ t)t')$.
\end{proposition}
We say that 
$t:\widetilde{I} \to I$ 
is a parameter change if $t'(u) \not=0$ for all $u \in \widetilde{I}$. 
\par
Two map germs 
$f:(\R^{n},x_{0})\to (\R^{p},y_{0})$ 
and $g:(\R^{n},x_{1})\to (\R^{p},y_{1})$ 
are said to be $\A$-equivalent
if there exist diffeomorphism germs 
$\phi:(\R^{n},x_{0})\to (\R^{n},x_{1})$ 
and
$\Phi \colon (\R^p,y_{1})\to(\R^p,y_{0})$ 
such that 
$f=\Phi \circ g\circ \phi$.
When $\Phi$ is the identity in the above,
$f$ and $g$ are said to be $\mathcal{R}$-equivalent. 
\par
A one-variable function $f:(I,t_0) \to (\R,0)$ has type $A_k$ at $t_0 \in I$ if $f^{(i)}(t_0)=0$ 
for $i=0,\dots,k$ and $f^{(k+1)}(t_0) \not=0$. 
Then $f$ is $\mathcal{R}$-equivalent to 
$g:(\R,0)\ni t\mapsto \pm t^{k+1}\in(\R,0)$
 (cf.\ \cite{Bruce-Giblin}).
\par
In this paper, if $f:(I,t_0) \to (\R,0)$ has type $A_k$ at $t_0 \in I$, we also say the {\it contact order of $f$ at $t_0$} is $k+1$ and denote it by ord$(f; t_0)=k+1$ (cf.\ \cite{Bruce-Giblin}).

\begin{proposition}[\cite{Nakatsuyama-Takahashi-2024}]\label{nm-type-parametrization}
Let $\gamma:(I,t_0) \to (\R^2,0)$ be a smooth map germ with $\gamma(t)=(x(t),y(t))$ and $1 \le n < m$. 
Suppose that $x$ has type $A_{n-1}$ and $y$ has type $A_{m-1}$ at $t_0$. 
Then $\gamma$ is $\mathcal{R}$-equivalent to of type $(n,m)$.
\end{proposition}
\par
We consider how to change the curvature of Legendre curves by a diffeomorphism of the target.
\begin{proposition}[\cite{Nakatsuyama-Takahashi-2024}]\label{diffeomorphism-target}
Let $(\gamma,\nu):I \to \R^2 \times S^1$ be a Legendre curve with curvature $(\ell,\beta)$ and 
$\Phi:\R^2 \to \R^2, \Phi(x,y)=(\phi_1(x,y),\phi_2(x,y))$ be a diffeomorphism. 
We denote $\gamma(t)=(x(t),y(t)), \nu(t)=(a(t),b(t))$ and $\widetilde{\gamma}=\Phi \circ \gamma$.
Then $(\widetilde{\gamma}, \widetilde{\nu}):I \to \R^2 \times S^1$ is a Legendre curve, where $\widetilde{\nu}=\overline{\nu}/|\overline \nu|$,
$$
\overline{\nu}(t)=(\phi_{2y}(\gamma(t))a(t)-\phi_{2x}(\gamma(t))b(t),-\phi_{1y}(\gamma(t))a(t)+\phi_{1x}(\gamma(t))b(t)).
$$
The curvature $(\widetilde{\ell},\widetilde{\beta})$ is given by 
\begin{align*} 
\widetilde{\ell}(t) &= \frac{1}{|\overline{\nu}(t)|^2}\Bigl( \bigl((\phi_{2xx}(\gamma(t))b^2(t)-2\phi_{2xy}(\gamma(t))a(t)b(t)\\
& \quad +\phi_{2yy}(\gamma(t))a^2(t))(-\phi_{1x}(\gamma(t))b(t)+\phi_{1y}(\gamma(t))a(t))\\
& \quad -\left(\phi_{1xx}(\gamma(t))b^2(t)-2\phi_{1xy}(\gamma(t))a(t)b(t)+\phi_{1yy}(\gamma(t))a^2(t)\right)\\
& \qquad (-\phi_{2x}(\gamma(t))b(t)+\phi_{2y}(\gamma(t))a(t)) \bigr)\beta(t)+(\phi_{1x}(\gamma(t))\phi_{2y}(\gamma(t))-\phi_{2x}(\gamma(t))\phi_{1y}(\gamma(t)))\ell(t) \Bigr),\\
\widetilde{\beta}(t)&=|\overline{\nu}(t)|\beta(t).
\end{align*}
\end{proposition}
\par
As special cases, we consider affine transformations and the reflection of the target of $\gamma$.
\begin{corollary}[\cite{Nakatsuyama-Takahashi-2024}]\label{change-target}
Let $(\gamma,\nu):I \to \R^2 \times S^1$ be a Legendre curve with curvature $(\ell,\beta)$ and $\Phi:\R^2 \to \R^2$ be a diffeomorphism. 
We denote $\nu=(a,b)$ and $\widetilde{\gamma}=\Phi \circ \gamma$.
\par
$(1)$ Suppose that the diffeomorphism $\Phi:\R^2 \to \R^2$ is given by  $\Phi(x,y)=(a_{11}x+a_{12}y,a_{21}x+a_{22}y)$, where $a_{11}a_{22}-a_{12}a_{21} \not=0$ and $a_{11},a_{12},a_{21},a_{22} \in \R$. 
Then $(\widetilde{\gamma},\widetilde{\nu}):I \to \R^2 \times S^1$ is a Legendre curve with the curvature $(\widetilde{\ell},\widetilde{\beta})=\left((a_{11}a_{22}-a_{12}a_{21})\ell/|\overline{\nu}|^2,|\overline{\nu}|\beta\right)$, where $\widetilde{\nu}=\overline{\nu}/|\overline{\nu}|$ and 
$\overline{\nu}=(a_{22}a-a_{21}b,-a_{12}a+a_{11}b)$.
\par
$(2)$ Suppose that the diffeomorphism $\Phi: \R^2 \to \R^2$ is given by $\Phi(x,y)=(y,x)$.  
Then $(\widetilde{\gamma},\widetilde{\nu}):I \to \R^2 \times S^1$ is a Legendre curve with the curvature $(\widetilde{\ell},\widetilde{\beta})=\left(-\ell,\beta\right)$, where $\widetilde{\nu}=(-b,-a)$.
\end{corollary}
By a direct calculation, we have the following.
\begin{proposition}[\cite{Fukunaga-Takahashi-2013}]\label{nu-change}
Let $(\gamma,\nu):I \to \R^2 \times S^1$ be a Legendre curve with curvature $(\ell,\beta)$. 
Then $(\gamma,-\nu):I \to \R^2 \times S^1$ is also a Legendre curve with curvature $(\ell,-\beta)$.
\end{proposition}
\par
For $n \in \mathbb{N} \cup \{0\}$, we say that a Legendre curve $(\gamma, \nu) : [a,b] \rightarrow \mathbb{R}^2 \times S^1$ is $C^{n}$-\emph{closed} if $(\gamma^{(k)}(a), \nu^{(k)}(a))=(\gamma^{(k)}(b), \nu^{(k)}(b))$ for all $k=0,\cdots,n$, where $\gamma^{(k)}(a)$, $\nu^{(k)}(a)$,  $\gamma^{(k)}(b)$ and $\nu^{(k)}(b)$ mean one-sided $k$-th differential.
We say that  a Legendre curve $(\gamma, \nu) : [a,b]
\rightarrow \mathbb{R}^2 \times S^1$ is 
$C^{\infty}$-\emph{closed} if $(\gamma^{(k)}(a), \nu^{(k)}(a))=
(\gamma^{(k)}(b), \nu^{(k)}(b))$ for all $ k \in \mathbb{N} \cup \{0\}$. 
In this paper, we say that $(\gamma, \nu)$ is a \emph{closed} Legendre curve if the
curve is $C^{\infty}$-closed (cf.\ \cite{Fukunaga-Takahashi-2016}). 
When $a$ and $b$ are singular points (respectively, inflection points) of $\gamma$, we treat these singular points (respectively, inflection points) as one singular point (respectively, inflection point). 

\section{Curvature equivalence}

We introduce an equivalence relation for Legendre curves. 
Let $I$ and $J$ be intervals of $\R$. 
Let $(\gamma,\nu): I \to \R^2 \times S^1$ and $(\widetilde{\gamma},\widetilde{\nu}): J \to \R^2 \times S^1$ be Legendre curves with curvatures $(\ell,\beta)$ and $(\widetilde{\ell},\widetilde{\beta})$, respectively.
\begin{definition}{\rm 
We say that two Legendre curves $(\gamma,\nu):I \to \R^2 \times S^1$ and 
$(\widetilde{\gamma},\widetilde{\nu}):J \to \R^2 \times S^1$ are {\it curvature equivalent} if the curvatures $(\ell,\beta)$ and $(\widetilde{\ell},\widetilde{\beta})$ are $\mathcal{S}$-$\mathcal{K}$-equivalent ({\it strictly $\mathcal{K}$-equivalent}), that is, 
there exist nowhere-zero smooth functions $\lambda_1,\lambda_2:J \to \R$ 
and a diffeomorphism $\phi:J \to I$ such that $(\widetilde{\ell},\widetilde{\beta})=(\lambda_1 \ell \circ \phi, \lambda_2 \beta \circ \phi)$.
}
\end{definition}
It is easy to show that curvature equivalence is an equivalent relation of Legendre curves.
By definition, curvature equivalence preserves the inflection and singular points of frontals.

\begin{remark}{\rm 
If $(\ell,\beta)$ and $(\widetilde{\ell},\widetilde{\beta})$ are $\mathcal{S}$-$\mathcal{K}$-equivalent, then $(\ell,\beta)$ and $(\widetilde{\ell},\widetilde{\beta})$ are $\mathcal{K}$-equivalent. 
Moreover, $\ell$ and $\widetilde{\ell}$ (respectively, $\beta$ and $\widetilde{\beta}$) are $\mathcal{K}$-equivalent.
}
\end{remark}

\begin{example}[Type $(n,m)$]\label{nm-type-curvature-equivalent}
{\rm 
Let $n, k \in \mathbb{N}$ and $m=n+k$. 
We consider $\gamma$ is of type $(n,m)$ (Example \ref{nm-type}).
Let $(\gamma,\nu):(\R,0) \to \R^2 \times S^1$,
\begin{align*}
\gamma(t) &=(\pm t^n,t^mf(t)),\ 
\nu(t) =\frac{1}{\sqrt{(mt^kf(t)+t^{k+1}\dot{f}(t))^2+n^2}}(-mt^kf(t)-t^{k+1}\dot{f}(t),\pm n),
\end{align*}
where $f:(\R,0) \to \R$ is a smooth function germ with $f(0) \not=0$. 
Then the curvature 
\begin{align*}
(\ell(t),\beta(t))=\biggl(&\pm \frac{nt^{k-1}(mkf(t)+(m+k+1)t\dot{f}(t)+t^{2}\ddot{f}(t))}{(mt^kf(t)+t^{k+1}\dot{f}(t))^2+n^2},\\
&-t^{n-1}\sqrt{(mt^kf(t)+t^{k+1}\dot{f}(t))^2+n^2}\biggr)
\end{align*}
is $\mathcal{S}$-$\mathcal{K}$-equivalent to $(t^{k-1},t^{n-1})$ around $0$.
}
\end{example}

\begin{remark}{\rm
It is well known that the two curves of type $(2,3)$ in Figure \ref{fig:23cusp} 
are diffeomorphic. 
From Example \ref{nm-type-curvature-equivalent}, they are also curvature equivalent. 
On the other hand, it is well known that the two curves of type $(4,5)$ in Figure \ref{fig:45cusp} 
are not diffeomorphic (cf.\ \cite{Bruce-Gaffney}). 
But from Example \ref{nm-type-curvature-equivalent}, they are curvature equivalent. 
Furthermore, types $(2,3)$ and $(4,5)$ are neither diffeomorphic, nor curvature equivalent. 
Therefore, the curvature equivalence preserves cusp types around the points.
}
\end{remark}

\begin{proposition}\label{invariants}
Let $(\gamma,\nu):I \to \R^2 \times S^1$ be a Legendre curve with curvature $(\ell,\beta)$. 
We denote $\nu=(a,b)$.
\par
$(1)$ Let $t:\widetilde{I} \to I$ be a parameter change, where $\widetilde{I}$ is an interval, and   $(\widetilde{\gamma},\widetilde{\nu})=(\gamma \circ t, \nu \circ t): \widetilde{I} \to \R^2 \times S^1$. 
Then $(\gamma,\nu)$ and $(\widetilde{\gamma},\widetilde{\nu})$ are curvature equivalent.
\par
$(2)$ Let $\Phi:\R^2 \to \R^2$ be $\Phi(x,y)=(a_{11}x+a_{12}y,a_{21}x+a_{22}y)$, where $a_{11}a_{22}-a_{12}a_{21} \not=0$ and $a_{11},a_{12},a_{21},a_{22} \in \R$ and let 
$(\widetilde{\gamma},\widetilde{\nu}):I \to \R^2 \times S^1$ be
$\widetilde{\gamma}=\Phi \circ \gamma$, $\widetilde{\nu}=\overline{\nu}/|\overline{\nu}|$ and 
$\overline{\nu}=(a_{22}a-a_{21}b,-a_{12}a+a_{11}b)$. 
Then $(\gamma,\nu)$ and $(\widetilde{\gamma},\widetilde{\nu})$ are curvature equivalent.
\par
$(3)$ Let $\Phi: \R^2 \to \R^2$ be $\Phi(x,y)=(y,x)$ and $(\widetilde{\gamma},\widetilde{\nu}):I \to \R^2 \times S^1$ be $\widetilde{\gamma}=\Phi \circ \gamma$, $\widetilde{\nu}=(-b,-a)$.
Then $(\gamma,\nu)$ and $(\widetilde{\gamma},\widetilde{\nu})$ are curvature equivalent. 
\par
$(4)$ $(\gamma,\nu)$ and $(\gamma,-\nu)$ are curvature equivalent. Moreover, $(\gamma,\nu)$ and $(-\gamma,\nu)$ are also curvature equivalent.
\end{proposition}
\demo
$(1)$ By Proposition \ref{Legendre.function}, the curvature of $(\widetilde{\gamma},\widetilde{\nu})$ is given by $((\ell \circ t) t', (\beta \circ t) t')$. 
Then it is $\mathcal{S}$-$\mathcal{K}$-equivalent to $(\ell,\beta)$.
Hence, $(\gamma,\nu)$ and $(\widetilde{\gamma},\widetilde{\nu})$ are curvature equivalent.
\par
$(2)$ By Corollary \ref{change-target} $(1)$, the curvature of $(\widetilde{\gamma},\widetilde{\nu})$ is given by $((a_{11}a_{22}-a_{12}a_{21})\ell/|\overline{\nu}|^2,|\overline{\nu}|\beta)$. 
Hence, $(\gamma,\nu)$ and $(\widetilde{\gamma},\widetilde{\nu})$ are curvature equivalent.
\par
$(3)$ By Corollary \ref{change-target} $(2)$, the curvature of $(\widetilde{\gamma},\widetilde{\nu})$ is given by $(-\ell,\beta)$. 
Hence, $(\gamma,\nu)$ and $(\widetilde{\gamma},\widetilde{\nu})$ are curvature equivalent.
\par
$(4)$ By Proposition \ref{nu-change}, the curvature of $(\gamma,-\nu)$ is given by $(\ell,-\beta)$. 
Hence, $(\gamma,\nu)$ and $(\gamma,-\nu)$ are curvature equivalent. 
Moreover, $(-\gamma,\nu):I \to \R^2 \times S^1$ is a Legendre curve with curvature $(\ell,-\beta)$. 
Hence, $(\gamma,\nu)$ and $(-\gamma,\nu)$ are curvature equivalent. 
\enD

\subsection{Local classifications of Legendre curves under the curvature equivalence}

We give a local classification of Legendre curves by the curvature equivalence 
under the condition that $\gamma$ is a finite type.

\begin{theorem}[Local classifications]\label{local-classifications}
Let $\gamma:(I,t_0) \to (\R^2,0)$ be a smooth map germ with $\gamma(t)=(x(t),y(t))$ and $n, m \in \mathbb{N}$. Suppose that $x$ has type $A_{n-1}$ and $y$ has type $A_{m-1}$ at $t_0$.
\par
$(1)$ Suppose that $n < m=n+k$, where $k \in \mathbb{N}$.  
Then there exists a smooth map $\nu:(I,t_0) \to S^1$ such that 
the Legendre curve $(\gamma,\nu):(I,t_0) \to \R^2 \times S^1$ is curvature equivalent to  $(\widetilde{\gamma},\widetilde{\nu}):(\R,0) \to \R^2 \times S^1$, 
$$
\widetilde{\gamma}(t)=(t^n,t^m), \ \widetilde{\nu}(t)=\frac{1}{\sqrt{m^2t^{2k}+n^2}}(-mt^k,n).
$$
\par
$(2)$ Suppose that $n = m$. Then there exist a smooth function germ $g: (\R,0) \to (\R,0)$ and a  non-zero constant $c \in \R$ such that $\gamma$ is $\mathcal{R}$-equivalent to $(\R,0) \to (\R,0), t \mapsto (\pm t^n, t^n (c+g(t)))$.
\par
$(i)$ If $g=0$,  
then there exists a smooth map $\nu:(I,t_0) \to S^1$ such that 
the Legendre curve $(\gamma,\nu):(I,t_0) \to \R^2 \times S^1$ is curvature equivalent to  $(\widetilde{\gamma},\widetilde{\nu}):(\R,0) \to \R^2 \times S^1$, 
$$
\widetilde{\gamma}(t)=(t^n,t^n), \ \widetilde{\nu}(t)=(-1,1).
$$
\par
$(ii)$ If $g$ has type $A_{p-1}$, where $p \in \mathbb{N}$,  
then there exists a smooth map $\nu:(I,t_0) \to S^1$ such that 
the Legendre curve $(\gamma,\nu):(I,t_0) \to \R^2 \times S^1$ is curvature equivalent to  $(\widetilde{\gamma},\widetilde{\nu}):(\R,0) \to \R^2 \times S^1$, 
$$
\widetilde{\gamma}(t)=(t^n,t^n(1+t^p)), \ 
\widetilde{\nu}(t)=\frac{1}{\sqrt{(n+nt^p+pt^{p})^2+n^2}}(-(n(1+t^p)+pt^{p}),n).
$$
\par
$(3)$ Suppose that $n=m+k > m$, where $k \in \mathbb{N}$.  
Then there exists a smooth map $\nu:(I,t_0) \to S^1$ such that 
the Legendre curve $(\gamma,\nu):(I,t_0) \to \R^2 \times S^1$ is curvature equivalent to  $(\widetilde{\gamma},\widetilde{\nu}):(\R,0) \to \R^2 \times S^1$, 
$$
\widetilde{\gamma}(t)=(t^n,t^m), \ \widetilde{\nu}(t)=\frac{1}{\sqrt{m^2+n^2t^{2k}}}(m,-nt^k).
$$
\end{theorem}
\demo
$(1)$ By Proposition~\ref{nm-type-parametrization}, 
$\gamma$ is $\mathcal{R}$-equivalent to of type $(n,m)$. 
By Proposition~\ref{invariants} (1) and Example \ref{nm-type}, there exists a smooth map $\nu:I \to S^1$ such that $(\gamma,\nu)$ and $(\overline{\gamma},\overline{\nu}):(\R,0) \to \R^2 \times S^1$,
\begin{align*}
\overline{\gamma}(t) &=(\pm t^n,t^mf(t)),\ 
\overline{\nu}(t) =\frac{1}{\sqrt{(mt^kf(t)+t^{k+1}\dot{f}(t))^2+n^2}}(-mt^kf(t)-t^{k+1}\dot{f}(t),\pm n),
\end{align*}
are curvature equivalent. 
Here, $f:(\R,0) \to \R$ is a function germ with $f(0) \not=0$. 
By Example \ref{nm-type-curvature-equivalent}, the curvature $(\overline{\ell},\overline{\beta})$ of $(\overline{\gamma},\overline{\nu})$ is $\mathcal{S}$-$\mathcal{K}$-equivalent to $(t^{k-1},t^{n-1})$. 
By a direct calculation, the curvature $(\widetilde{\ell},\widetilde{\beta})$ of the Legendre curve $(\widetilde{\gamma},\widetilde{\nu}):(\R,0) \to \R^2 \times S^1$, 
$$
\widetilde{\gamma}(t)=(t^n,t^m), \ \widetilde{\nu}(t)=\frac{1}{\sqrt{m^2t^{2k}+n^2}}(-mt^k,n)
$$
is also $\mathcal{S}$-$\mathcal{K}$-equivalent to $(t^{k-1},t^{n-1})$. 
It follows that $(\overline{\gamma},\overline{\nu})$ and $(\widetilde{\gamma},\widetilde{\nu})$ are curvature equivalent. 
Hence, $(\gamma,\nu)$ and $(\widetilde{\gamma},\widetilde{\nu})$ are curvature equivalent. 
\par
$(2)$ 
By a similar argument of 
Proposition~\ref{nm-type-parametrization}, 
we can show that there exist a function germ $g: (\R,0) \to (\R,0)$ 
and a non-zero constant $c \in \R$ such that 
$\gamma$ is $\mathcal{R}$-equivalent to $(\R,0) \to (\R,0), t \mapsto 
\left(
\pm t^n, t^n (c+g(t))
\right)$.
\par
$(i)$ Since $\gamma$ is $\mathcal{R}$-equivalent to $(\R,0) \to (\R^2,0), t \mapsto (\pm t^n, ct^n)$,  
there exists a smooth map $\nu:(I,t_0) \to S^1$ such that $(\gamma,\nu)$ and $(\overline{\gamma},\overline{\nu}):(\R,0) \to \R^2 \times S^1$,
$$
\overline{\gamma}(t)=(\pm t^n,c t^n), \ \overline{\nu}(t)=\frac{1}{\sqrt{c^2+1}}(-c,\pm 1)
$$
are curvature equivalent. 
The curvature $(\overline{\ell},\overline{\beta})$ of $(\overline{\gamma},\overline{\nu})$ is given by $(\overline{\ell},\overline{\beta})=(0,-n\sqrt{c^2+1}t^{n-1})$.
Hence, $(\overline{\ell},\overline{\beta})$ is $\mathcal{S}$-$\mathcal{K}$-equivalent to $(0,t^{n-1})$. 
By a direct calculation, the curvature $(\widetilde{\ell},\widetilde{\beta})$ of the Legendre curve $(\widetilde{\gamma},\widetilde{\nu}):(\R,0) \to \R^2 \times S^1$, 
$$
\widetilde{\gamma}(t)=(t^n,t^n), \ \widetilde{\nu}(t)=(-1,1)
$$
is also $\mathcal{S}$-$\mathcal{K}$-equivalent to $(0,t^{n-1})$. 
It follows that $(\overline{\gamma},\overline{\nu})$ and $(\widetilde{\gamma},\widetilde{\nu})$ are curvature equivalent. 
Hence, $(\gamma,\nu)$ and $(\widetilde{\gamma},\widetilde{\nu})$ are curvature equivalent.
\par
$(ii)$ Since $g$ has type $A_{p-1}$, we can denote $g(t)=t^p h(t)$, where $h:(\R,0) \to \R$ is a smooth function with $h(0) \not=0$. 
Since $\gamma$ is $\mathcal{R}$-equivalent to $(\R,0) \to (\R^2,0), t \mapsto (\pm t^n, t^n (c+g(t)))$,  
there exists a smooth map $\nu:(I,t_0) \to S^1$ such that $(\gamma,\nu)$ and $(\overline{\gamma},\overline{\nu}):(\R,0) \to \R^2 \times S^1$,
\begin{align*}
\overline{\gamma}(t)&=(\pm t^n,  t^n (c+g(t))), \ 
\overline{\nu}(t)
=
\frac{1}{\sqrt{
(n(c+g(t))+t\dot{g}(t))^2+n^2}}
\left(
-(n(c+g(t))+t\dot{g}(t)),\pm n
\right)
\end{align*}
are curvature equivalent. 
The curvature $(\overline{\ell},\overline{\beta})$ of $(\overline{\gamma},\overline{\nu})$ is given by 
\begin{align*}
&(\overline{\ell}(t),\overline{\beta}(t))
=
\left(
\pm\frac{n((n+1)\dot{g}(t)+t\ddot{g}(t))}{(n(c+g(t))+t\dot{g}(t))^2+n^2}, 
-t^{n-1}\sqrt{(n(c+g(t))+t\dot{g}(t))^2+n^2} 
\right).
\end{align*}
Moreover, 
\begin{align*}
&(n+1)\dot{g}(t)+t\ddot{g}(t)\\
=&(n+1)(pt^{p-1}h(t)+t^p \dot{h}(t))+t(p(p-1)t^{p-2}h(t)+2pt^{p-1}\dot{h}(t)+t^p\ddot{h}(t))\\
=&\left( ((n+1)p+p(p-1))h(t)+(n+1+2p)t \dot{h}(t)+t^2\ddot{h}(t) \right)t^{p-1}.
\end{align*}
Hence, $(\overline{\ell},\overline{\beta})$ is $\mathcal{S}$-$\mathcal{K}$-equivalent to $(t^{p-1},t^{n-1})$. 
By a direct calculation, the curvature $(\widetilde{\ell},\widetilde{\beta})$ of the Legendre curve $(\widetilde{\gamma},\widetilde{\nu}):(\R,0) \to \R^2 \times S^1$, 
\begin{align*}
\widetilde{\gamma}(t)&=(t^n,t^n(1+t^p)), \ 
\widetilde{\nu}(t)=\frac{1}{\sqrt{(n+nt^p+pt^{p})^2+n^2}}(-(n(1+t^p)+pt^{p}),n)
\end{align*}
is also $\mathcal{S}$-$\mathcal{K}$-equivalent to $(t^{p-1},t^{n-1})$. 
It follows that $(\overline{\gamma},\overline{\nu})$ and $(\widetilde{\gamma},\widetilde{\nu})$ are curvature equivalent. 
Hence, $(\gamma,\nu)$ and $(\widetilde{\gamma},\widetilde{\nu})$ are curvature equivalent.
\par
$(3)$ By the same arguments as in $(1)$, there exists a smooth map $\nu:I \to S^1$ such that $(\gamma,\nu)$ and $(\overline{\gamma},\overline{\nu}):(\R,0) \to \R^2 \times S^1$,
\begin{align*}
\overline{\gamma}(t) &=(t^n f(t), \pm t^m),\  
\overline{\nu}(t) =\frac{1}{\sqrt{(nt^kf(t)+t^{k+1}\dot{f}(t))^2+m^2}}(\pm m,-nt^kf(t)-t^{k+1}\dot{f}(t)),
\end{align*} 
are curvature equivalent. 
Here, $f:(\R,0) \to \R$ is a smooth function germ with $f(0) \not=0$.   
The curvature $(\overline{\ell},\overline{\beta})$ of $(\overline{\gamma},\overline{\nu})$ is $\mathcal{S}$-$\mathcal{K}$-equivalent to $(t^{k-1},t^{m-1})$. 
By a direct calculation, the curvature $(\widetilde{\ell},\widetilde{\beta})$ of the Legendre curve $(\widetilde{\gamma},\widetilde{\nu}):(\R,0) \to \R^2 \times S^1$, 
$$
\widetilde{\gamma}(t)=(t^n,t^m), \ \widetilde{\nu}(t)=\frac{1}{\sqrt{n^2t^{2k}+m^2}}(m,-n t^k)
$$
is also $\mathcal{S}$-$\mathcal{K}$-equivalent to $(t^{k-1},t^{m-1})$. 
It follows that $(\overline{\gamma},\overline{\nu})$ and $(\widetilde{\gamma},\widetilde{\nu})$ are curvature equivalent. 
Hence, $(\gamma,\nu)$ and $(\widetilde{\gamma},\widetilde{\nu})$ are curvature equivalent. 
\enD

\begin{remark}{\rm
In Theorem \ref{local-classifications} $(2)$ $(i)$, 
$(\widetilde{\gamma},\widetilde{\nu})$ is also curvature equivalent to 
$(\hat{\gamma}(t),\hat{\nu}(t))=((t^n,0),(0,1))$ by Proposition~\ref{invariants} (2). 
In Theorem \ref{local-classifications} $(2)$ $(ii)$, 
$(\widetilde{\gamma},\widetilde{\nu})$ is also curvature equivalent to 
$(\hat{\gamma}(t),\hat{\nu}(t))=((t^n,t^{n+p}),(-(n+p)t^p,n)/{\sqrt{(n+p)^2t^{2p}+n^2}})$. 
In Theorem \ref{local-classifications} $(3)$, $(\widetilde{\gamma},\widetilde{\nu})$ is also curvature equivalent to 
$$
\left(\hat{\gamma}(t),\hat{\nu}(t)\right)=\left((t^m,t^n),(-nt^k,m)/{\sqrt{n^2t^{2k}+m^2}}\right)
$$ 
by Propositions~\ref{invariants} (3) and (4).  
This Legendre curve $(\hat{\gamma},\hat{\nu})$ is 
just the Legendre curve $(\widetilde{\gamma},\widetilde{\nu})$ in 
Theorem~\ref{local-classifications} $(1)$. 
}
\end{remark}

\subsection{Global classifications of Legendre curves under the curvature equivalence}

Let $(\gamma,\nu):I \to \R^2 \times S^1$ be a Legendre curve with curvature $(\ell,\beta)$.
We consider $Z(\ell)$ and $Z(\beta)$ as the sets of zero points of $\ell$ and $\beta$, that is, the sets of inflection points and singular points.
We denote 
$Z(\ell) =\{s_1, \dots, s_m\}$ with $s_1<\cdots<s_m$, 
$Z(\beta)=\{t_1, \dots, t_n\}$ with $t_1<\cdots<t_n$.
Let $(u_1,\dots,u_{m+n})$ be arranged in ascending order, 
including the equal sign of $Z(\ell) \cup Z(\beta)$.
We say that the ordered tuple $(u_1,\dots,u_{m+n})$ 
is the {\it sequence order of the set of zero points $(\ell,\beta)$}. 

\begin{theorem}[Global classifications]\label{global-classifications}
Let $(\gamma,\nu):I \to \R^2 \times S^1$, $(\widetilde{\gamma},\widetilde{\nu}):\widetilde{I} \to \R^2 \times S^1$ be Legendre curves with curvatures $(\ell,\beta)$, $(\widetilde{\ell},\widetilde{\beta})$, respectively. 
Suppose that the numbers of inflection points and singular points are finite and the same, that is,  
$Z(\ell) =\{s_1,\dots, s_m\}$ with $s_1<\cdots<s_m$, 
$Z(\beta)=\{t_1,\dots, t_n\}$ with $t_1<\cdots<t_n$, 
$Z(\widetilde{\ell}) =\{\widetilde{s}_1, \dots, \widetilde{s}_m\}$ with $\widetilde{s}_1<\cdots<\widetilde{s}_m$, 
$Z(\widetilde{\beta})=\{\widetilde{t}_1, \dots, \widetilde{t}_n\}$ with $\widetilde{t}_1<\cdots<\widetilde{t}_n$. 
Moreover, the sequence orders $(u_1,\dots,u_{m+n})$ of the set of zero points $(\ell,\beta)$ and $(\widetilde{u}_1,\dots,\widetilde{u}_{m+n})$ of the set of zero points $(\widetilde{\ell},\widetilde{\beta})$ are the same, that is, 
if $u_k=s_i$ (respectively $u_k=t_j$), then $\widetilde{u}_k=\widetilde{s}_i$ (respectively, $\widetilde{u}_k=\widetilde{t}_j$) for all $k=1,\dots,m+n$. 
Further, the contact orders of $\ell$ at $s_i$ and of $\widetilde{\ell}$ at $\widetilde{s}_i$, the contact orders of $\beta$ at $t_j$ and of $\widetilde{\beta}$ at $\widetilde{t}_j$ are the same for all $i=1,\dots,m$, $j=1,\dots,n$.  
Then $(\gamma,\nu)$ and $(\widetilde{\gamma},\widetilde{\nu})$ are curvature equivalent. 
\end{theorem}
In order to prove Theorem \ref{global-classifications}, 
we prepare the following two Lemmas.

\begin{lemma}\label{Leibnizrule}
Let $f:I \to \R$ be a smooth function, 
$u\in I$ be a zero point of $f$ 
and $\varphi:\widetilde{I}\to I$ be a parameter change. 
Then, the contact order of $f$ at $u$ and 
that of $(f\circ\varphi)\varphi'$ at $\tilde{u}$ coincide. 
Here, $\widetilde{u}\in\widetilde{I}$ satisfies with 
$\varphi(\widetilde{u})=u$.
\end{lemma}
This lemma is a direct consequence of the Liebnitz rule. 
Therefore, we omit the proof of this lemma.
\begin{lemma}\label{function}
Let $f,g: I \to \R$ be smooth functions. 
Suppose that the zero points of $f$ and $g$ are finite and the same $Z(f)=Z(g)=\{t_1,\dots,t_n\}$ with $t_1 <\dots <t_n$. 
Moreover, the contact orders of $f$ and of $g$ at $t_i$ are the same. 
Then there exists a nowhere-zero smooth function 
$\lambda:I \to \R$ such that $f(t)=\lambda(t) g(t)$ for all $t \in I$.
\end{lemma}
\demo
Suppose that $I=[a,b]$ and $a<t_1<\dots<t_n<b$. 
We can also prove that the other cases, $a=t_1$ or $b=t_n$ or $(a,b]$ or $[a,b)$ or open interval $(a,b)$, by the same arguments. 
If $t \not \in Z(f)=Z(g)$, we define a smooth function $\lambda(t)=f(t)/g(t)$. 
If $t_i \in Z(f)=Z(g)$, by the same contact order 
$\mathrm{ord}(f; t_{i})=\mathrm{ord}(g;t_{i})=r$, 
there exist smooth function germs $\widetilde{f}, \widetilde{g}:(I,t_i) \to \R$ with 
$\widetilde{f}(t_i) \not=0$ and $\widetilde{g}(t_i) \not=0$ such that 
$f(t)=(t-t_i)^r\widetilde{f}(t)$ and $g(t)=(t-t_i)^r \widetilde{g}(t)$ around $t_i$. 
Moreover, we define a smooth function $\widetilde{\lambda}(t)=\widetilde{f}(t)/\widetilde{g}(t)$ around $t_i$. 
Then $\lambda(t)=\widetilde{\lambda}(t)$ around $t_i$ except for $t_i$. 
Hence, we can extend $\lambda$ to smoothly at $t_i$. 
Since $t_i$ is finite and the construction, there exists a nowhere-zero smooth function $\lambda:I \to \R$ such that $f(t)=\lambda(t) g(t)$ for all $t \in I$.
\enD
\par
{\it Proof of Theorem \ref{global-classifications}}. 
By using a partition of unity, there exists a parameter change, 
$\varphi: \widetilde{I} \to I$ with $\varphi'(u)>0$ such that $\varphi(\widetilde{u}_k)=u_k$ for all $k=1,\dots,m+n$. 
By Proposition \ref{invariants} $(1)$ 
and Lemma~\ref{Leibnizrule}, 
we may assume that $\widetilde{I}=I$ and 
$u_k=\widetilde{u}_k$ for all $k=1,\dots,m+n$. 
By assumptions and Lemma \ref{function}, there exist nowhere-zero smooth functions 
$\lambda_1,\lambda_2:I \to \R$ such that $\ell(t) =\lambda_1(t) \widetilde{\ell}(t)$ and $\beta(t) =\lambda_2(t) \widetilde{\beta}(t)$ for all $t \in I$.
It follows that $(\ell,\beta)$ and $(\widetilde{\ell},\widetilde{\beta})$ are $\mathcal{S}$-$\mathcal{K}$-equivalent. 
Hence, $(\gamma,\nu)$ and $(\widetilde{\gamma},\widetilde{\nu})$ are curvature equivalent. 
\enD
\par
According to Theorem \ref{global-classifications}, if the sequence order of the set of zero points $(\ell,\beta)$ and the contact order are known, then $(\gamma,\nu)$ can be uniquely determined up to curvature equivalence. 

\begin{remark}{\rm 
$(1)$ If $(\gamma,\nu)$ is a closed Legendre curve, then in the statement of 
Theorem~\ref{global-classifications}, 
the condition that 
``if $u_k=s_i$ (respectively $u_k=t_j$), then 
$\widetilde{u}_k=\widetilde{s}_{i}$ 
(respectively, $\widetilde{u}_k=\widetilde{t}_j$) 
for all $k=1,\dots,m+n$'' 
can be replaced by the condition that 
``if $u_k=s_i$ (respectively $u_k=t_j$),
 then  $\widetilde{u}_{k+l}=\widetilde{s}_i$ 
 (respectively, $\widetilde{u}_{k+l}=\widetilde{t}_j$) for all $k=1,\dots,m+n$, where $1 \le l \le m+n$ and $k+l$ mod $m+n$''. 
If the contact orders are the same as the corresponding points, then the assertion of Theorem \ref{global-classifications} holds. 
\par
$(2)$ The parameter change $t \mapsto -t$ is corresponding to the condition that 
``if $u_k=s_i$ (respectively $u_k=t_j$), then $\widetilde{u}_{m+n+1-k}=\widetilde{s}_{m+1-i}$ 
(respectively, $\widetilde{u}_{m+n+1-k}=\widetilde{t}_{n+1-j}$) 
for all $k=1,\dots,m+n$''. 
If the contact orders are the same as the corresponding points, then the assertion of Theorem \ref{global-classifications} holds. 
}
\end{remark}

\begin{proposition}\label{number}
Let $(\gamma,\nu): I \to \R^2 \times S^1$ be a closed Legendre curve with curvature $(\ell,\beta)$ 
that satisfies the conditions of Theorem~\ref{global-classifications}. 
Then the number of points with odd contact orders of $\ell$ is even. 
Also, the number of points with odd contact orders of $\beta$ is even.
\end{proposition}
\demo 
Suppose that 
$I=[a,b]$ and $Z(\ell)=\{s_{1},\ldots,s_{m}\}$ 
with 
$s_{1}<\cdots <s_{m}$. 
We may assume that 
$a<s_{1}$ and $s_{m}<b$. 
We set that 
$\widetilde{s}_{0}=a,
\tilde{s}_{j}=(s_{j}+s_{j+1})/2$ 
and 
$\widetilde{s}_{m}=b$ 
for $j=1,\ldots,m-1$. 
Since $\ell(a)=\ell(b)$, 
we have 
\begin{equation*}
\ell(\widetilde{s}_{0})\ell(\widetilde{s}_{1})^{2}
\cdots\ell(\widetilde{s}_{m-1})^{2}\ell(\widetilde{s}_{m})
=
\left(
\ell(\widetilde{s}_{0})\ell(\widetilde{s}_{1})
\right)
\cdots
\left(
\ell(\widetilde{s}_{m-1})\ell(\widetilde{s}_{m})
\right)
>0.
\end{equation*}
On the other hand, 
$\mathrm{ord}(\ell;s_{j})$ is odd 
if and only if 
$\ell(\widetilde{s}_{j-1})\ell(\widetilde{s}_{j})<0$ holds. 
Therefore, 
the number of points with odd contact orders of $\ell$ is even. 
For the case of $\beta$, the proof is the same. 
\enD

We give concrete examples of Legendre curves. See \cite{Fukunaga-Takahashi-2016, Nakatsuyama-Takahashi-2024} for more examples.
\begin{example}\label{example_ab}{\rm
Let $a,b \in \mathbb{N}$ and $a \not=b$. Suppose that there are no pair $(n,m) \in \mathbb{Z} \times \mathbb{Z}$ with $0 \le 1+2n < 4a$ and $0 \le 1+2m <4b$ such that $b(1+2n)=a(1+2m)$. 
Let $(\gamma[a,b],\nu[a,b]):[0,2\pi) \to \R^2 \times S^1$ be 
\begin{align*}
\gamma[a,b](t)&=(\sin at, \sin bt), \
\nu[a,b](t)=\frac{1}{\sqrt{a^2 \cos^2 at+b^2 \cos^2 bt}}\left(-b\cos bt, a\cos at\right).
\end{align*}
By assumption, $\nu[a,b]$ is a $C^\infty$ mapping, that is, $a^2 \cos^2 at+b^2 \cos^2 bt \not=0$ for all $t \in [0,2\pi)$. 
Since $\dot{\gamma}[a,b](t)=(a \cos at,b \cos bt)$, $\dot{\gamma}[a,b](t) \cdot \nu[a,b](t)=0$ for all $t \in [0,2\pi)$. 
It follows that $(\gamma[a,b],\nu[a,b])$ is a Legendre curve. 
By $\mu[a,b](t)=\left(-a\cos at,-b\cos bt\right)/\sqrt{a^2 \cos^2 at+b^2 \cos^2 bt}$ and a direct calculation, the curvature $(\ell[a,b],\beta[a,b])$ of $(\gamma[a,b],\nu[a,b])$ is given by
\begin{align*}
&(\ell[a,b](t),\beta[a,b](t))=
\left(\frac{-ab(b\cos at\sin bt-a\sin at\cos bt)}{a^2 \cos^2 at+b^2 \cos^2 bt}, 
-\sqrt{a^2 \cos^2 at+b^2 \cos^2 bt}\right).
\end{align*}
Note that $(\ell[a,b],\beta[a,b])$ is $\mathcal{S}$-$\mathcal{K}$-equivalent to $(b\cos at\sin bt-a\sin at\cos bt,1)$.
}
\end{example}

\begin{example}\label{example_n}{\rm
Let $n \in \mathbb{N}$ and $n\ne 1$. 
Let $(\gamma[n],\nu[n]):[0,2\pi) \to \R^2 \times S^1$ be 
\begin{align*}
\gamma[n](t)&=(n\cos t-\cos nt, n\sin t-\sin nt), \ 
\nu[n](t)=\left(\sin \frac{n+1}{2}t, -\cos \frac{n+1}{2}t\right).
\end{align*}
Since $\dot{\gamma}[n](t)=(-n\sin t+n\sin nt,n\cos t-n\cos nt)$, 
\begin{align*}
\dot{\gamma}[n](t) \cdot \nu[n](t)
&=
\left(
\sin \frac{n+1}{2}t
\right)
(-n\sin t+n\sin nt)
-\left(
\cos \frac{n+1}{2}t
\right)
(n\cos t-n\cos nt) \\
&=-n\left(\cos\left(\frac{n+1}{2}-1\right)t-\cos\left(\frac{n+1}{2}-n\right)t\right)\\
&=-n\left(\cos\frac{n-1}{2}t-\cos\frac{n-1}{2}t\right)\\
&=0
\end{align*}
for all $t \in [0,2\pi)$. 
It follows that $(\gamma[n],\nu[n])$ is a Legendre curve. 
By 
$$
\mu[n](t)=\left(\cos \frac{(n+1)}{2} t,\sin \frac{(n+1)}{2} t \right)
$$ 
and a direct calculation, the curvature $(\ell[n],\beta[n])$ of $(\gamma[n],\nu[n])$ is given by
$$
(\ell[n](t),\beta[n](t))=\left(\frac{n+1}{2},2n\sin\frac{n-1}{2}t\right).
$$
Note that $(\ell[n],\beta[n])$ is $\mathcal{S}$-$\mathcal{K}$-equivalent to $(1,\sin ((n-1)t/2))$.
}
\end{example}

\begin{example}\label{example_m}{\rm
Let $m \in \mathbb{N}$. 
Let $(\gamma[m],\nu[m]):[0,2\pi) \to \R^2 \times S^1$ be 
\begin{align*}
\gamma[m](t)&=(m\sin t-\sin mt, m \cos t+\cos mt), \ 
\nu[m](t)=\left(-\cos \frac{m-1}{2}t, -\sin \frac{m-1}{2}t\right).
\end{align*}
Since $\dot{\gamma}[m](t)=(m\cos t-m\cos mt,-m\sin t-m\sin mt)$, 
\begin{align*}
\dot{\gamma}[m](t) \cdot \nu[m](t)
&=
-\left(\cos \frac{m-1}{2}t\right)(m\cos t-m\cos mt)
-\left(\sin \frac{m-1}{2}t\right)(-m\sin t-m\sin mt) \\
&=m\left(-\cos\left(\frac{m-1}{2}+1\right)t+\cos\left(\frac{m-1}{2}-m\right)t\right)\\
&=m\left(-\cos\frac{m+1}{2}t+\cos\frac{m+1}{2}t\right)\\
&=0
\end{align*}
for all $t \in [0,2\pi)$. 
It follows that $(\gamma[m],\nu[m])$ is a Legendre curve. 
By 
$$
\mu[m](t)=\left(\sin \frac{(m-1)}{2} t, -\cos \frac{(m-1)}{2} t\right)
$$ 
and a direct calculation, the curvature $(\ell[m],\beta[m])$ of $(\gamma[m],\nu[m])$ is given by
$$
(\ell[m](t),\beta[m](t))=\left(\frac{m-1}{2},2m\sin\frac{m+1}{2}t\right).
$$
Note that $(\ell[m],\beta[m])$ is $\mathcal{S}$-$\mathcal{K}$-equivalent to 
$(0, \sin t)$ if $m=1$, and to $(1,\sin ((m+1)t/2))$ if $m \ge 2$.
}
\end{example}
\par
If $n=m+2$ and $m \ge 2$, then the Legendre curves $(\gamma[n],\nu[n])$ as in Example \ref{example_n} and $(\gamma[m],\nu[m])$ as in Example \ref{example_m} are curvature equivalent by the direct definition. 
It is also follows from Theorem \ref{global-classifications} (see Figure \ref{fig:3}).

\begin{figure}[htbp]
\centering
\begin{minipage}{0.3\textwidth}
\centering
\includegraphics[width=4cm,height=4cm]{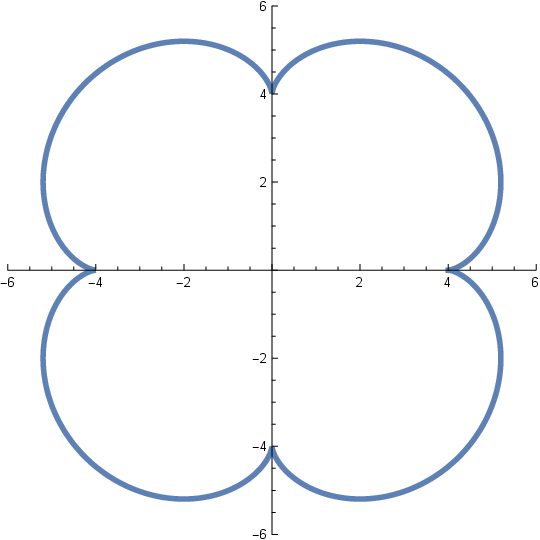}
\end{minipage}
\hspace{0.5cm}
\begin{minipage}{0.3\textwidth}
\centering
\includegraphics[width=4cm,height=4cm]{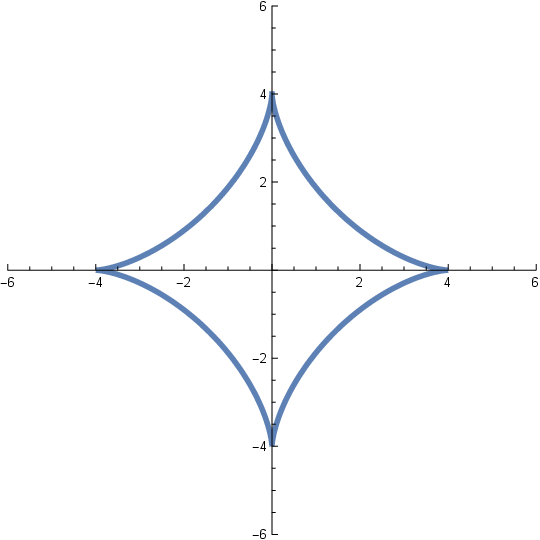}
\end{minipage}
\caption{On the left is the case of $n=5$ 
in Example \ref{example_n}, that is, $\gamma[5](t)= (5\cos t-\cos 5t,5\sin t-\sin 5t)$. 
On the right is the case of $m=3$ in Example \ref{example_m}, that is, $\gamma[3](t) = (3\sin t-\sin 3t, 3\cos t+\cos 3t)$.}
\label{fig:3}
\end{figure}


Nozomi Nakatsuyama, 
\\
Muroran Institute of Technology, Muroran 050-8585, Japan,
\\
E-mail address: 25096009b@muroran-it.ac.jp
\\
\\
Masatomo Takahashi, 
\\
Muroran Institute of Technology, Muroran 050-8585, Japan,
\\
E-mail address: masatomo@muroran-it.ac.jp
\\
\\
Minoru Yamamoto,
\\
Department of Mathematics, Faculty of Education, Hirosaki University, Hirosaki, Aomori, 036-8560, Japan
\\
E-mail address: minomoto@hirosaki-u.ac.jp


\begin{thebibliography}{99}
{\small

\bibitem{Arnold1}
V. I. Arnol'd,
\newblock{\it Singularities of Caustics and Wave Fronts}.
\newblock Mathematics and Its Applications {\bf 62} Kluwer Academic Publishers, 1990.

\bibitem{Arnold2}
V. I. Arnol'd, S. M. Gusein-Zade and A. N. Varchenko,
\newblock{\it Singularities of Differentiable Maps vol. {\rm I}}.
\newblock Birkh\"auser, 1986. 

\bibitem{Bruce-Gaffney} J. W. Bruce and T. J. Gaffney, 
\newblock{Simple singularities of mappings ${\mathbb C},0 \to {\mathbb C}^2,0$}. 
\newblock{\it J. London Math. Soc. (2).}
{\bf 26}, (1982), 465--474.

\bibitem{Bruce-Giblin} J. W. Bruce and P. J. Giblin, 
\newblock{\it Curves and singularities. A geometrical introduction to singularity theory. Second edition}. 
\newblock{Cambridge University Press}, Cambridge, 1992. 

\bibitem{doCarmo} M. P. do Carmo, 
\newblock{\it Differential geometry of curves and surfaces.} 
Translated from the Portuguese. Prentice-Hall, Inc., Englewood Cliffs, N.J., 1976. 

\bibitem{Dias-Tari} F. S. Dias and F. Tari,
\newblock{On vertices and inflections of plane curves}.
\newblock{\it J. Singul.} {\bf 17}, (2018), 70--80. 

\bibitem{Fukunaga-Takahashi-2013} T. Fukunaga and M. Takahashi,
\newblock{Existence and uniqueness for Legendre curves}.
\newblock{\it J. Geom.} {\bf 104}, (2013), 297--307.

\bibitem{Fukunaga-Takahashi-2016} T. Fukunaga, M. Takahashi, 
\newblock{On convexity of simple closed frontals}.
\newblock{\it Kodai Math. J.} {\bf 39}, (2016), 389--398. 

\bibitem{Gibson} C. G. Gibson,
\newblock{\it Elementary geometry of differentiable curves. An undergraduate introduction}.
\newblock{Cambridge University Press}, Cambridge, 2001.

\bibitem{Gray} A. Gray, E. Abbena and S. Salamon, 
\newblock{\it Modern differential geometry of curves and surfaces with Mathematica. Third edition.} 
\newblock{Studies in Advanced Mathematics. Chapman and Hall/CRC}, Boca Raton, FL, 2006.

\bibitem{Ishikawa-book} G. Ishikawa, 
\newblock{\it Singularities of Curves and Surfaces in Various Geometric Problems}.
\newblock{CAS Lecture Notes 10, Exact Sciences}, 2015.

\bibitem{Ishikawa-2018} G. Ishikawa, 
\newblock{Singularities of frontals.} 
\newblock{\it Singularities in generic geometry}, 55--106, 
{\it Adv. Stud. Pure Math.,} {\bf 78}, Math. Soc. Japan, Tokyo, 2018. 

\bibitem{Ishikawa-2020} G. Ishikawa, 
\newblock{Recognition problem of frontal singularities.} 
\newblock{\it J. Singul.} {\bf 21}, (2020), 149--166.

\bibitem{Izumiya-book} S. Izumiya, M. C. Romero-Fuster, M. A. S. Ruas and F. Tari,
\newblock{\it Differential Geometry from a Singularity Theory Viewpoint}.
\newblock{World Scientific Pub. Co Inc.} 2015.

\bibitem{Nakatsuyama-Takahashi-2024} N. Nakatsuyama and M. Takahashi, 
\newblock{On vertices of frontals in the Euclidean plane}.
\newblock{\it Bull. Braz. Math. Soc. (N.S.)}  {\bf 55}, (2024), Paper No. 35, 21 pp.

\bibitem{Salarinoghabi-Tari} M. Salarinoghabi and F. Tari,
\newblock{Flat and round singularity theory of plane curves.} 
\newblock{\it Q. J. Math.}, {\bf 68}, (2017), 1289--1312.
}
\end{thebibliography}
\end{document}